\documentclass[a4paper,11pt]{article}
\usepackage[T1]{fontenc} 
\usepackage{amsmath,amsthm}
\usepackage[francais,english]{babel}
\usepackage[dvips]{graphicx}
\usepackage{todonotes}
\usepackage{color}
\usepackage{soul}
\usepackage{amsfonts,amssymb}
\usepackage{geometry}
\usepackage{hyperref}
\usepackage{stmaryrd}
\usepackage{fancyhdr}
\usepackage{url}
\usepackage{dsfont}
 \usepackage[latin1]{inputenc} 

\theoremstyle{definition}

\newtheorem*{Mainthm}{Theorem {\ref{main}}}
\newtheorem*{cormain}{Corollary {\ref{lmodcor}}}

\newtheorem*{thm*}{Theorem}

\newtheorem*{acknowledgements}{Acknowledgements}
\newtheorem{prop}{Proposition}[section]

\newtheorem{LM}[prop]{Lemma}
\newtheorem{thm}[prop]{Theorem}
\newtheorem{df}[prop]{Definition}
\newtheorem{cor}[prop]{Corollary}

\newtheorem{rem}[prop]{Remark}
\newtheorem*{rem*}{Remark}

\newcommand{\of}{\mathfrak{o}}
\newcommand{\p}{\mathfrak{p}}
\newcommand{\w}{\varpi}

\newcommand{\C}{\mathbb{C}}

\newcommand{\J}{\mathcal{J}}

\newcommand{\Z}{\mathbb{Z}}
\newcommand{\1}{\mathbf{1}}

\def\End{\mathrm{End}}
\def\GL{\mathrm{GL}}
\def\Hom{\mathrm{Hom}}

\def\dim{\mathrm{dim}}

\def\bJ{\mathbf{J}}
\def\Ql{\overline{\mathbb{Q}_{\ell}}}
\def\Aut{\mathrm{Aut}}
\def\Zl{\overline{\mathbb{Z}_{\ell}}}
\def\Fl{\overline{\mathbb{F}_{\ell}}}

\def\Mat{\mathrm{Mat}}
\def\Vect{\mathrm{Vect}}

\def\ind{\mathrm{ind}}

\def\Ind{\mathrm{Ind}}

\def\presuper#1#2%
  {\mathrm{}%
   \mathrmen{\vphantom{#2}}^{#1}%
   \kern-\scriptspace%
   #2}

\begin{document}

\title{Test vectors for local cuspidal Rankin--Selberg integrals}

\author{Robert KURINCZUK\footnote{Robert Kurinczuk, Heilbronn Institute for Mathematical Research, Imperial College London, United Kingdom. Email: robkurinczuk@gmail.com}, Nadir MATRINGE\footnote{Nadir Matringe, Universit\'e de Poitiers, Laboratoire de Math\'ematiques et Applications,
T\'el\'eport 2 - BP 30179, Boulevard Marie et Pierre Curie, 86962, Futuroscope Chasseneuil Cedex. Email: Nadir.Matringe@math.univ-poitiers.fr}}


\maketitle

\begin{abstract}
Let~$\pi_1,\pi_2$ be a pair of cuspidal complex, or~$\ell$-adic, representations of the general linear group of rank~$n$ over a non-archimedean local field~$F$ of residual characteristic~$p$, different to~$\ell$.  Whenever the local Rankin--Selberg~$L$-factor~$L(X,\pi_1,\pi_2)$ is nontrivial, we exhibit explicit test vectors in the Whittaker models of~$\pi_1$ and~$\pi_2$ such that the local Rankin--Selberg integral associated to these vectors and to the characteristic function of~$\of_F^n$ is equal to~$L(X,\pi_1,\pi_2)$. As an application we prove that the~$L$-factor of a pair of banal~$\ell$-modular cuspidal representations is the reduction modulo~$\ell$ of the~$L$-factor of any pair of~$\ell$-adic lifts.
\end{abstract}

\section{Introduction} 

The integral representation of local~$L$-factors, of pairs of complex irreducible representations 
of general linear groups over a non-archimedean local field~$F$, was developed in the fundamental paper~\cite{JPS2} of Jacquet--Piatetski-Shapiro--Shalika. These~$L$-factors are Euler factors which are the greatest common divisors, in a certain sense, of families of integrals~$I$ of Whittaker functions.  For~$n\geqslant m$, as a by-product of the definition, if~$\pi_1$ and~$\pi_2$ are irreducible smooth complex (or~$\ell$-adic) representations of~$\GL_n(F)$ and~$\GL_m(F)$ respectively with Whittaker models~$W(\pi_1,\psi)$ and~$W(\pi_2,\psi^{-1})$, extended to all irreducible representations via the Langland's classification, then it is known that there is a finite number~$r$ of Whittaker functions~$W_i\in W(\pi_1,\psi)$ and~$W_i'\in W(\pi_2,\psi^{-1})$, and a finite number of Schwartz functions~$\Phi_i$ on~$F^n$ if~$n=m$, such that the~$L$-factor~$L(X,\pi_1,\pi_2)$ can be expressed as~$\sum_{i=1}^r I(X,W_i,W_i')$, or~$\sum_{i=1}^r I(X,W_i,W_i',\Phi_i)$ when~$n=m$. A natural question which thus arises is whether one can find an explicit family of such \emph{test vectors}. 

A famous instance of test vectors is the essential vectors for generic representations (cf.~\cite{JPSConducteur}, \cite{JacquetCorrection}, \cite{Matringe}). It is shown in these references that these vectors are test vectors for~$L(X,\pi_1,\pi_2)$ when~$\pi_1$ is a generic representation of~$\GL_n(F)$,~$\pi_2$ is an unramified standard module of~$\GL_m(F)$, and~$n>m$.  

Interesting partial results have been obtained in~\cite{CogdellsStudent}, and, as indicated in~\cite{CogdellsStudent}, the theory of derivatives and its interpretation in terms of restriction of Whittaker functions (cf.~\cite{CP},~\cite{Matringe}) should reduce the general problem to the cuspidal case. Here, we establish the cuspidal case: that for pairs of cuspidal representations~$\pi_1$ and~$\pi_2$, we can choose~$r=1$, and moreover, we exhibit explicit test vectors, in the interesting case, whenever~$L(X,\pi_1,\pi_2)$ is not equal to one.  The fact that~$r$ can be chosen to be~$1$ when~$L(X,\pi_1,\pi_2)=1$, for any pair of irreducible representations~$\pi_1$ and~$\pi_2$ of~$\GL_n(F)$ and~$\GL_m(F)$, is explained in the proof of~\cite[Theorem~2.7]{JPS2} and follows from standard facts on Kirillov models. We do not provide completely explicit test functions in this case, possibly a quite technical problem, and we in fact do not consider this case in the sequel, as it is not needed for our application to reduction modulo~$\ell$. 

Before we state our main theorem, we explain our normalisation of Haar measure (Section \ref{SectionHaar}), as for our application to reduction modulo~$\ell$ some care needs to be taken with the normalisation.  Let~$\mathfrak{o}_F$ denote the ring of integers of~$F$ with unique maximal ideal~$\mathfrak{p}_F$, and let~$q$ denote the cardinality of the residue field~$\mathfrak{o}_F/\mathfrak{p}_F$ and~$p$ its characteristic.  We fix our Haar measure on~$\GL_n(F)$ to give the pro-$p$ unipotent radical~$K_n^1$ of~$\GL_n(\mathfrak{o}_F)$ volume one.  It will turn out that this is a good choice of normalisation for reduction modulo~$\ell$ for primes~$\ell$ not equal to~$p$ because~$K_n^1$ is a pro-$p$ subgroup.  In particular, the volume of any pro-$p$ subgroup of~$\GL_n(F)$ which occurs in our computation will be a power of~$q$.  

Now we state our main theorem. Let~$\pi_1$ and~$\pi_2$ be cuspidal complex, or~$\ell$-adic, representations of~$\GL_n(F)$ such that~$L(X,\pi_1,\pi_2)$ is non-trivial, so that~$\pi_2\simeq \chi\pi_1^\vee$ for some unramified character~$\chi$ of~$F^\times$.  Let~$e$ denote the common ramification index of~$\pi_1$ and~$\pi_2$ (see Section~\ref{SimpleTypes}). We denote by~$W_1$ and~$W_2$ the explicit Whittaker functions for~$\pi_1$ and~$\pi_2$, as constructed in~\cite{PS}, with respect to a suitable nondegenerate character of the standard maximal unipotent subgroup of~$\GL_n(F)$ and suitable maximal extended simple types in~$\pi_1$ and~$\pi_2$. 

\begin{Mainthm}
There is an integer~$r$ such that
\begin{align*}
I(X,W_1,W_2,1_{\of_F^n})&=(q-1)(q^{n/e}-1)q^r\frac{1}{1-(\nu(\w_F)X)^{n/e}}\\
&=(q-1)(q^{n/e}-1)q^r L(X,\pi_1,\pi_2).
\end{align*}
\end{Mainthm} 
The factor~$q^r$ occurs in our computation as a product of volumes, with respect to certain quotient measures, of quotients of pro-$p$ subgroups related to the groups of Bushnell--Kutzko \cite{BK93} in their explicit construction of~$\pi_1,\pi_2$.  Clearly, after our computation we could simply renormalise our measure by the factor~$(q-1)(q^{n/e}-1)q^r$ and under the new normalisation have an equality between the integral and the~$L$-factor, hence~$(W_1,W_2,1_{\mathfrak{o}_F^n})$ is a test vector in the sense described earlier.  However, it is important to keep track of these factors for our application to reduction modulo~$\ell$.

We now describe the proof of this theorem.  In Section \ref{PS-basis}, we carefully choose an appropriate basis of~$F^n$ and simple types in our cuspidal representations, so that the subgroup of~$\GL_n(F)$ defined by these simple types decomposes well with respect to the Iwasawa decomposition and satisfies some other important properties (see Proposition \ref{properties}). In Section \ref{SectWhittaker}, we analyse the support of the explicit Whittaker functions of Paskunas and Stevens in terms of this well chosen group (Proposition \ref{furtherproperties}). This preparation, which constitutes a substantial amount of the path to our main result, then allows us to compute the integral in Section \ref{Main}. 

Our interest in test vectors originated in the the study of~$\ell$-modular Rankin--Selberg~$L$-factors, for~$\ell\neq p$, as introduced in~\cite{KM14}. Let~$\pi_1$ and~$\pi_2$ be integral cuspidal~$\ell$-adic representations of~$\GL_n(F)$ and~$\GL_m(F)$, and~$\tau_1=r_{\ell}(\pi_1)$ and~$\tau_2=r_{\ell}(\pi_2)$ their reductions modulo~$\ell$, which are cuspidal~$\ell$-modular representations.  By \cite[Theorem 3.13]{KM14}, the local factor~$L(X,\tau_1,\tau_2)$ always divides~$r_\ell(L(X,\pi_1,\pi_2))$.  In particular,~$L(X,\tau_1,\tau_2)=r_\ell(L(X,\pi_1,\pi_2))$ whenever~$L(X,\pi_1,\pi_2)=1$.  Hence the interesting case, where a strict division can happen is when~$L(X,\pi_1,\pi_2)$ is not equal to~$1$, and, in particular,~$n=m$.  In~\cite{MZeta}, it was shown that for banal representations the~$\ell$-modular Godement--Jacquet~$L$-factor is equal to the reduction modulo~$\ell$ of the~$\ell$-adic Godement--Jacquet~$L$-factor. It is thus natural to ask:  if~$\pi_1$ and~$\pi_2$ are~$\ell$-adic integral cuspidal representations of~$\GL_n(F)$ with banal reductions~$\tau_1$ and~$\tau_2$, does one have~$L(X,\tau_1,\tau_2)=r_\ell(L(X,\pi_1,\pi_2))$? As a corollary of our main result on test vectors applied to~$\ell$-adic Rankin-Selberg integrals, we answer this question in the affirmative.  

\begin{cormain}
Let~$\tau_1$ and~$\tau_2$ be two banal cuspidal~$\ell$-modular representations of~$\GL_n(F)$, and~$\pi_1$ and~$\pi_2$ be any cuspidal~$\ell$-adic lifts, then
\[L(X,\tau_1,\tau_2)=r_\ell(L(X,\pi_1,\pi_2)).\]
\end{cormain}

In~\cite{KM14}, this corollary plays a key role in the classification of~$L$-factors of generic~$\ell$-modular representations, and their relationship with~$\ell$-adic~$L$-factors via reduction modulo~$\ell$.  

It would be interesting to pursue the methods of this paper for integral representations of other~$L$-factors, such as the Asai, exterior square, and symmetric square~$L$-factors.

\section{Notations}

Let~$F$ be a non-archimedean local field of residual characteristic~$p$ and residual cardinality~$q$. Throughout,~$R$ will denote one of the fields~$\mathbb{C}$,~$\Ql$, and~$\Fl$ and we assume that $\ell\neq p$.  For~$E$ any extension of~$F$, we denote by~$\of_E$ the ring of integers of~$E$; by~$\w_E$ a uniformiser of~$E$; by~$\p_E=(\w_E)$ the unique maximal ideal of~$\of_E$; by~$q_E$ the residual cardinality of~$E$; and let~$|~|_{E,R}:E^\times\rightarrow R^\times$ denote the unramified character defined by~$|\w_E|_{E,R}=q_E^{-1}$, thus~$|~|_{E,\mathbb{C}}$ is the restriction of the absolute value to~$E^\times$ normalised in the usual way. When the field~$R$ considered is clear we will remove the index~$R$ from~$|~|_{E,R}$, and when~$E=F$ we will remove the index~$F$ as well.

Let~$G_n=\GL_n(F)$,~$K_n=\GL_n(\of_F)$,~$K_n^1=1+\Mat_{n,n}(\p_F)$, and let~$Z_n$ be the centre of~$G_n$. For~$g$ in~$G_n$, by abuse of notation, we denote by~$|g|$ the quantity~$|\det(g)|$.
Put~$\eta_n=\begin{pmatrix}0&\cdots&0&1\end{pmatrix}\in\Mat_{1,n}(F)$, and let~$P_n$ be the standard mirabolic subgroup of~$G_n$, i.e. the set of all matrices~$g$ in~$G_n$ such that~$\eta_n g=\eta_n$.  Let~$N_n$ be the unipotent radical of the standard Borel subgroup of upper triangular matrices in~$G_n$.  For~$k\in\mathbb{Z}$, let~$G_n^{(k)}=\{g\in G_n:~|g|_F=q^{-k}\}$. For any subset~$X$ of~$G_n$, let~$X^{(k)}=X\cap G_n^{(k)}$, and let~$\mathbf{1}_X$ denote the characteristic function of~$X$. 

Let~$\Ql$ denote an algebraic closure of the~$\ell$-adic numbers,~$\Zl$ denote its ring of integers, and~$\Fl$ denote its residue field which is an algebraic closure of the finite field of~$\ell$-elements.

\section{Representations with coefficients in~$R$}\label{coeffinR}

We only consider smooth~$R$-representations, that is smooth representations with coefficients in~$R$, and we use~$\vee$ as an exponent to denote the contragredient.  We call a representation on a~$\Ql$-vector space an~\emph{$\ell$-adic representation}, and a representation on an~$\Fl$-vector space an~\emph{$\ell$-modular representation}. Let~$(\pi,\mathcal{V})$ be an irreducible~$\ell$-adic representation of~$G_n$. We call~$\pi$ \emph{integral} if~$\mathcal{V}$ contains a~$G_n$-stable~$\Zl$-lattice.  Notice that for an~$\ell$-adic character~$\nu:G_n\rightarrow\Ql^\times$ this just means that~$\nu$ takes values in~$\Zl^\times$.

An~$R$-representation is called \emph{cuspidal} if it is irreducible and never appears as a quotient of a properly parabolically induced representation.  By~\cite[II~4.12]{V}, a cuspidal~$\ell$-adic representation is integral if and only if its central character is integral, hence~the contragredient of a cuspidal~$\ell$-adic representation~$\pi$ is integral if and only if~$\pi$ is integral.  Let~$\pi$ be an integral cuspidal~$\ell$-adic representation and~$\mathfrak{L}$ be a~$G_n$-stable~$\Zl$-lattice in the space of~$\pi$.  Let~$r_{\mathfrak{L}}(\pi)$ be the~$\ell$-modular representation induced on the space~$\mathfrak{L}\otimes_{\Zl}\Fl$. This~$\ell$-modular representation is also cuspidal (and irreducible) by~\cite[III~5.10]{V}, and hence independent of the choice of the lattice~$\mathfrak{L}$ by the Brauer--Nesbitt principle~\cite[Theorem~1]{VigInt}, we thus write~$r_{\ell}(\pi)$ for~$r_{\mathfrak{L}}(\pi)$ and call~$r_{\ell}(\pi)$ the \emph{reduction modulo~$\ell$} of~$\pi$.  We also say that~$\pi$ \emph{lifts}~$r_{\ell}(\pi)$, and it follows from~\cite[III~5.10]{V} that all cuspidal~$\ell$-modular representations lift to cuspidal~$\ell$-adic representations. Following~\cite[Remark 8.15]{MSDuke}, we call a cuspidal~$\ell$-modular representation~$\tau$ \emph{banal} if~$\tau\not\simeq \tau\otimes |~|_F$ (notice that the definition in Remark 8.15 of \cite{MSDuke} refers to a condition given in Proposition 8.9 of this reference, which in the cuspidal case reduces to the condition we give here). For~$H$ a closed subgroup of~$G$, we write~$\Ind_H^G$ for the functor of smooth induction taking representations of~$H$ to representations of~$G$, and write~$\ind_H^G$ for the functor of smooth induction with compact support. 

\section{Normalisation of Haar measures}\label{SectionHaar}

We now discuss our normalisation of Haar measures. The basic reference 
for~$R$-Haar measures is~\cite[I~2]{VigInt}, but we also refer the reader to~\cite[Section~2.2]{KM14} for more details on the splitting of Haar measures with respect to standard decompositions. Let~$dg$ be the~Haar measure on~$G_n$ normalised to give~$K_n^1$ volume~$1$.  

We normalise the right Haar measure on~$P_n$ so that~$dp\left(P_n\cap K_n^1\right)=1$, on 
$N_n$ so that~$dn\left(N_n\cap K_n^1\right)=1$, and on~$Z_n$ so that~$dz\left(Z_n\cap K_n^1\right)=1$.  For the remainder of this section, let~$G$ denote a closed subgroup of~$G_n$ with Haar measure~$d_Gg$.  For any open subgroup~$U$ of~$G$, we define the Haar measure~$d_U g$ on~$U$ as the restriction of~$d_Gg$, in particular~$d_Ug$ is normalised as soon as~$d_Gg$ is.

If~$H$ is a closed subgroup of~$G$ with right Haar measure~$d_Hh$, and such that the modulus character of~$G$ restricts to~$H$ as the modulus character of~$H$, we descend~$d_Gg$ to a right-invariant measure~$d_{H\backslash G}g$ on~$H\backslash G$ as explained in~\cite[I~2.8]{VigInt}. For~$f$ a smooth map from~$G$ to~$R$ with compact support, denoting by~$f^H$ the map on~$H\backslash G$ defined by 
\[f^{H}(g)=\int_{H} f(hg)d_H h,\] the usual relation is satisfied:
\[ \int_{H\backslash G}f^{H}(g)d_{H\backslash G}g=\int_G f(g)d_G g.\] 
This implies that~$d_{H\backslash G}g$ is normalised as soon as~$d_Gg$ and~$d_Hg$ are.

Indeed, if~$K$ is a compact subgroup of~$G$, applying the equality above to~$f=\1_K$, so that 
\[f^H=d_H(K\cap H)\1_{H\backslash HK}\] gives the relation
\begin{equation}\label{quotient}d_G(K)=d_{H\backslash G}(H\backslash HK)d_H(K\cap H).\end{equation} 
This gives for example the normalisation \[d_{H\backslash G}(H\backslash HK_n^1)= d_H(H\cap K_n^1)\backslash d_G(G\cap K_n^1).\]
With these normalisations, we have  the splitting \[dg=|p|_F^{-1}dpdzdk.\]
This splitting descends on~$N_n\backslash G_n$, in which case~$dg$ denotes the normalised right invariant measure on~$N_n\backslash G_n$ and~$dp$ the right invariant measure on $N_n\backslash P_n$.
Notice that with such normalisations, the volume of all pro-$p$ subgroups of~$G_n$, of~$P_n$ and of~$Z_n$ will be (positive or negative) powers of~$q$.  Moreover, for such choices, reduction modulo~$\ell$ commutes with integration (cf. \cite[Remark~2.1]{KM14}), i.e. if~$f\in \mathcal{C}_c^\infty(X,\Zl)$ for~$X$ equal to~$G_n$ or any of the homogeneous spaces~$K\backslash L$ with~$L$ a subgroup of~$G_n$ considered above, then~$\int_X f(x)dx\in \Zl$, and~\[r_{\ell}\left(\int_X f(x)dx\right)=\int_X r_{\ell}(f(x))dx.\]
\textit{For the rest of this section, we suppose that~$R$ has characteristic zero, and we recall some classical equalities,which all follow from Relation (\ref{quotient}).} 

For a finite set~$A$, we let~$|A|$ denote its cardinality in~$R$.  Suppose that~$G=K$ compact, and~$U$ is an open subgroup of~$K$, then
\begin{equation} \label{cardinal}d_{U\backslash K}(U\backslash K)=\frac{d_K(K)}{d_K(U)}=|U\backslash K|\in R.\end{equation}
Finally, if~$V$ is a closed subgroup of~$K$ (using the fact that~$K$ is unimodular, hence that~$d_K(UV)=d_K(V^{-1}U^{-1})=d_K(VU)$), one obtains 
\begin{align}
\notag d_{V\backslash K}(V\backslash VU)&=\frac{d_K(VU)}{d_V(V)}=
\frac{d_K(UV)}{d_K(U)}\frac{d_V(V\cap U)}{d_V(V)}\frac{d_K(U)}{d_V(V\cap U)}\\
\label{equal}&=\frac{|U\backslash UV|}{|V\cap U\backslash V|}\frac{d_K(U)}{d_V(V\cap U)}=\frac{d_K(U)}{d_V(V\cap U)}= d_{V\cap U\backslash U}(V\cap U\backslash U).\end{align}

\textit{By convention, from now on, we will use the same letter for the measure on~$G$ and its descent to~$H\backslash G$ (and when the context is clear for its restriction to an open subgroup as well).}

\section{Rankin--Selberg integrals and local factors}\label{local factors}

Let~$\psi$ be an additive character of~$F$ which is trivial on~$\p_F$, but non-trivial on~$\of_F$.  By abuse of notation, also denote by~$\psi$ the nondegenerate character of~$N_n$ defined for~$x=(x_{i,j})\in N_n$ by
\[\psi(x)=\psi\left(\sum_{i=1}^{n-1}x_{i,i+1}\right),\] 
which is necessarily integral in the~$\ell$-adic case because~$N_n$ is exhausted by its pro-$p$ subgroups. If~$\pi$ is a cuspidal representation of~$G_n$, then it is \emph{generic} (cf.~\cite{BZ} in the complex or~$\ell$-adic case, and~\cite[III 5.10]{V} for~$R=\Fl$), meaning~$\dim(\Hom_{N_n}(\pi,\psi))=1$, and hence it has a unique \emph{Whittaker model}~$W(\pi,\psi)$, equal to the image of~$\pi$ in~$\Ind_{N_n}^{G_n}(\psi)$. Suppose that~$\pi$ is an integral cuspidal~$\ell$-adic representation of~$G_n$, then the~$\Zl$-submodule~$W_e(\pi,\psi)$ of~$W(\pi,\psi)$ consisting of all functions in~$W(\pi,\psi)$ which take values in~$\Zl$ is a~$G_n$-stable lattice in~$\pi$ (cf.~\cite[Theorem~2]{VigInt}). Then by definition~$r_{\ell}(\pi)\simeq W_e(\pi,\psi)\otimes_{\Zl}\Fl$, which is irreducible and cuspidal (cf.~Section 2.1 and the references given there). Thus~$W_e(\pi,\psi)\otimes_{\Zl}\Fl$ is a space of Whittaker functions for~$\pi$ with values in~$\Fl$, hence equal to~$W(r_{\ell}(\pi),r_{\ell}(\psi))$. For~$W\in W_e(\pi,\psi)$, we write~$r_{\ell}(W)$ for the image of~$W$ in~$W(r_{\ell}(\pi),r_{\ell}(\psi))$.

Finally, we recall the definition of the Rankin--Selberg local~$L$-factors for a pair of cuspidal~$R$-representations of~$G_n$.  The construction is originally due to Jacquet--Piatetski-Shapiro--Shalika \cite{JPS2} for complex representations, and works equally well for~$\Ql$-representations. This construction was extended to a construction for representations over any algebraically closed field of characteristic prime to~$p$ in \cite{KM14}.  As we are ultimately interested in~$\mathbb{C},\Ql$ and~$\Fl$ representations we give precise references to the construction in \cite{KM14}. 

Let~$\pi_1$ and~$\pi_2$ be cuspidal representations of~$G_n$,~$W_1\in W(\pi_1,\psi)$,~$W_2\in W(\pi_2,\psi^{-1})$, and~$\Phi\in\mathcal{C}_c^{\infty}(F^n)$ be a locally constant function from~$F^n$ to~$R$ with compact support. By \cite[Proposition 3.3]{KM14}, for~$k\in\mathbb{Z}$, the coefficients
\[c_k(W_1,W_2,\Phi)=\int_{N_n\backslash G_n^{(k)}}W_1(g)W_2(g)\Phi(\eta_n g)dg\]
are well defined and vanish for~$k$ sufficiently negative. In fact, these coefficients 
vanish for~$k$ sufficiently negative because both~$W_1$ and~$W_2$ vanish on~$P_n^{(k)}$ for such~$k$, as a consequence of \cite[Proposition 2.2]{JPS2}. Hence the \emph{local Rankin--Selberg integral}
\[I(X,W_1,W_2,\Phi)=\sum_{k\in\mathbb{Z}} c_k(W_1,W_2,\Phi)X^k\]
is a formal Laurent series with coefficients in~$R$.  In fact, by~\cite[Theorem~3.5]{KM14},~$I(X,W_1,W_2,\Phi)\in R(X)$ is a rational function, and as~$W_1$ varies in~$W(\pi_1,\psi)$,~$W_2$ varies in~$W(\pi_2,\psi^{-1})$, and~$\Phi$ varies in~$\mathcal{C}_c^{\infty}(F^n)$, the~$R$-submodule of~$R(X)$ spanned by~$I(X,W_1,W_2,\Phi)$ is a fractional ideal of $R[X^{\pm 1}]$, and has a unique generator~$L(X,\pi_1,\pi_2)$ which is an Euler factor. We call~$L(X,\pi_1,\pi_2)$ the \emph{local Rankin--Selberg~$L$-factor}, and note that it does not depend on the choice of the character~$\psi$. If~$R=\Ql$, it is shown in~\cite[Corollary~3.6]{KM14} that the~$L$-factor is the inverse of a polynomial in~$\Zl[X]$, and thus it makes sense to talk of its reduction modulo~$\ell$. Moreover, it follows from~\cite[Theorem~3.13]{KM14}, that if~$\pi_1$ and~$\pi_2$ are two integral cuspidal~$\ell$-adic representations of~$G_n$, then one has 
\[L(X,r_\ell(\pi_1),r_\ell(\pi_2))|r_{\ell}(L(X,\pi_1,\pi_2)) .\]
Now by~\cite[Proposition~8.1,~(ii)]{JPS2}, the~$L$-factor~$L(X,\pi_1,\pi_2)$ is equal to~$1$ unless 
$\pi_2\simeq \chi\pi_1^\vee$ for some unramified character~$\chi$ of~$F^\times$.  Hence if~$\pi_2\not\simeq \chi\pi_1^\vee$~then~$L(X,r_\ell(\pi_1),r_\ell(\pi_2))=r_{\ell}(L(X,\pi_1,\pi_2))=1$.

For our computations to come, we use a decomposition of the Rankin--Selberg integral in the special case where~$\pi_2\simeq\pi_1^\vee$, in particular their central characters are inverse of each other. Thus we assume this is the case for the rest of this section. For~$k\in\mathbb{Z}$, we set
\[b_k(W_1,W_2)=\int_{N_n\backslash P_n^{(k)}}W_1(p)W_2(p)dp,\] 
which, similarly to~$c_k$, vanishes for~$k$ sufficiently negative, and we put 
\[I_{(0)}(X,W_1,W_2)=\sum_{k\in\mathbb{Z}} b_k(W_1,W_2)q^kX^k.\]
 Let~$\Phi\in \mathcal{C}_c^\infty(F^n)$ be a~$K_n$-invariant function, for~$i\in\mathbb{Z}$, we set 
\[a_{ni}(\Phi)=\int_{z\in G_1^{(ni)}} \Phi(\eta_n z)dz,\]
which vanishes for~$i$ sufficiently negative, and we put \[Z(X,\Phi)=\sum_{i\in \Z}a_{ni}(\Phi)X^{ni}.\]
As~$G_n^{(k)}=\coprod_{i\in \Z}P_n^{(k-ni)}Z_n^{(ni)}K_n$, from the splitting of Section \ref{SectionHaar} we find \[c_k(W_1,W_2,\Phi)=\sum_{i\in \Z} a_{ni}(\Phi)q^{k-ni} \int_{(K_n\cap P_n)\backslash K_n}b_{k-ni}(\rho(k)W_1,\rho(k)W_2)dk,\]
from which we deduce
\[I(X,W_1,W_2,\Phi)=Z(X,\Phi)\left(\int_{(K_n\cap P_n)\backslash K_n}I_{(0)}(X,\rho(k)W_1,\rho(k)W_2)dk\right).\] 
Taking~$\Phi$ equal to the characteristic function~$\1_{\of_F^n}$, we obtain the formula 
\begin{equation}\label{I=LI_{(0)}}I(X,W_1,W_2,\1_{\of_F^n})=\frac{q-1}{1-X^n}\int_{(K_n\cap P_n)\backslash K_n}I_{(0)}(X,\rho(k)W_1,\rho(k)W_2)dk.\end{equation}
The equality~$Z(X,\1_{\of_F^n})=\frac{q-1}{1-X^n}$ is standard (cf.~\cite[Theorem~3.1]{MZeta}) except that in our setting, we get the extra constant~$q-1$ from our choice of normlisation on~$Z_n$, as we set~$dz(Z_n\cap K_n^1)=1$ instead of the usual~$dz(Z_n\cap K_n)=1$. 

\section{Simple types and reduction modulo~$\ell$}\label{SimpleTypes}
\textit{For this beginning of this section we assume that~$R=\mathbb{C}$ or~$\Ql$}. Let~$V$ be an~$n$-dimensional~$F$-vector space, let~$\End_F(V)$ denote the~$F$-algebra $\End_F(V)$ of~$F$-endomorphisms of~$V$ and let~$G$ denote the group~$\Aut_F(V)$ of~$F$-automorphisms of~$V$. Hence~$G$ identifies with~$G_n$ as soon as we choose a basis of~$V$.  In~\cite{BK93}, every cuspidal~$R$-representation of~$G$ is constructed explicitly as~$\ind_{\bJ}^G(\Lambda)$, where~$\bJ$ is an open and compact-mod-centre subgroup of~$G$, and~$\Lambda$ is an irreducible representation of~$\bJ$ of finite dimension.  The pairs~$(\bJ,\Lambda)$ are called \emph{extended maximal simple types}, and for any such pair~$\ind_{\bJ}^G(\Lambda)$ is (irreducible and) cuspidal by \cite[Chapter 6]{BK93}.  We briefly explain the construction of the group~$\bJ$, focusing on the properties which we shall use.  

An~$\of_F$-lattice chain~$\mathcal{L}$ in~$V$ is a non-empty set of~$\of_F$-lattices~$\{L_i:i\in\mathbb{Z}\}$ such that, for all~$i\in \mathbb{Z}$,~$L_{i+1} \subsetneq L_i$ and there exists~$e(\mathcal{L})\in\mathbb{Z}$ such that~$L_{i+e(\mathcal{L})}=\varpi_F L_i$.  The construction of \cite{BK93}, starts with data~$(\beta,\mathcal{L})$ called \emph{maximal simple strata} consisting of
\begin{enumerate}
\item an element~$\beta\in \End_F(V)$ which generates a simple field extension~$E=F[\beta]$;
\item an \emph{$\of_F$-lattice chain}~$\mathcal{L}$ in~$V$ such that~$E^\times \mathcal{L}\subset \mathcal{L}$ (i.e. for any $x\in E^\times$ and $L\in \mathcal{L}$ we have $xL\in \mathcal{L}$); in particular~$\mathcal{L}$ is an \emph{$\of_E$-lattice chain}, and it is required (as~$(\beta,\mathcal{L})$ is maximal) that~$L_{i+1}=\varpi_EL_i$; 
\end{enumerate}
which satisfy a technical condition (cf. \cite[1.5.5]{BK93} where the simple strata we consider are among those denoted $[\mathfrak{A},-,0,\beta]$).  

Let~$(\beta,\mathcal{L})$ be a maximal simple strata.  We denote by~$\mathfrak{A}=\mathfrak{A}(\mathcal{L})$ the~$\of_F$-order in~$\End_F(V)$ and~$\mathfrak{B}=\mathfrak{B}(\beta,\mathcal{L})$ the~$\of_E$-order in~$\End_E(V)$ defined by~$\mathcal{L}$,
\[\mathfrak{A}=\End_{\of_F}(\mathcal{L})=\bigcap_k \End_{\of_F}(L_k),\quad \mathfrak{B}=\mathfrak{B}(\beta,\mathcal{L})=\End_{\of_E}(\mathcal{L})=\End_{\of_E}(L_0) .\]
In \cite[3.1]{BK93} Bushnell--Kutzko define compact open subgroups of~$G$ denoted by~$H^1=H^1(\beta,\mathcal{L})$,~$J^1=J^1(\beta,\mathcal{L})$, and~$J=J(\beta,\mathcal{L})$. The properties we will need are:
\begin{enumerate}
\item the groups~$H^1\leqslant J^1$ are pro-$p$ (by definition), are normalised by~$E^\times$ and are normal subgroups of~$J$ by \cite[3.1.15]{BK93}, moreover $J\subset \Aut_{\of_F}(L_0)$ (by definition).
\item Put~$m=n/[E:F]$, by \cite[3.1.15]{BK93} we have
\[J=\mathfrak{B}^\times J^1,~\mathfrak{B}^\times\cap J^1= 1+\w_E \mathfrak{B}~\text{and}~J/J^1\simeq 
\mathfrak{B}^\times/(1+\w_E \mathfrak{B})\simeq G_m(k_E).\]
\end{enumerate}
We then set~$\bJ=\bJ(\beta,\mathcal{L})=E^\times J$, in particular~$\bJ$ is compact mod~$E^\times$ and hence compact mod~$F^\times$.  Notice that if~$\pi\simeq\ind_{\bJ}^G\Lambda$, the centre~$F^\times$ of~$G$ acts by the central character~$\omega_\pi$ of~$\pi$ through~$\Lambda$.  Finally, we note that the construction of~$\Lambda$ depends on our fixed additive character~$\psi$ (cf. \cite[3.2]{BK93}).

The definitions above do not include the groups of the maximal simple types for level zero cuspidal representations (see \cite[5.5.10 (b)]{BK93}), although these can be considered formally as part of the construction described above for the \emph{maximal zero strata}~$(0,\mathcal{L})$ with~$\beta=0$ and~$e(\mathcal{L})=1$.  In this case, we put~$J=\mathfrak{A}^\times$,~${\bJ}=F^\times J$,~$H^1=J^1=1+\varpi_F\mathfrak{A}$, and~$J/J^1=\mathfrak{A}^\times/(1+\varpi_F\mathfrak{A})\simeq G_n(k_F)$.

\textit{Now we consider~$\Fl$-representations.} It follows from \cite[Chapitre~IV]{V} that the Bushnell--Kutzko classification of cuspidal~$\Ql$-representations adapts well to~$\Fl$-representations.  We will only need to know the following facts:

Let~$\tau$ be a cuspidal~$\ell$-modular representation of~$G$.  As we recalled in Section 
\ref{coeffinR}, there exists an integral cuspidal~$\ell$-adic representation~$\pi$ such that~$\tau=r_\ell(\pi)$.  Choose an extended maximal simple type~$({\bJ},\Lambda)$ such that~$\pi\simeq \ind_{\bJ}^G(\Lambda)$, as in the beginning of this section.  A cuspidal~$\ell$-adic representation is integral if and only if its central character~$\omega_{\pi}$ is integral, by \cite[II 4.13]{V} (the direction integral implies integral central character being clear). We recall why this is true. First as~$\bJ$ is compact mod~$F^\times$, we claim that the irreducible representation~$\Lambda$ is integral if and only if~$\omega_\pi$ is integral. Again, one direction is clear. For the other,  suppose that~$\omega_\pi$ is integral and choose a random not necessarily~$\bJ$-stable lattice~$\mathfrak{L}_0$ in the space~$V_\Lambda$ of~$\Lambda$. It is stabilised by a compact open subgroup~$U$ of~$\bJ$, and choosing representatives~$c_1,\dots,c_r$ of 
$\bJ / F^\times U$, one has~$\Lambda(\bJ)(\mathfrak{L}_0)=\sum_{i=1}^r\Lambda(c_i)(\mathfrak{L}_0)$, hence~$\mathfrak{L}_{\Lambda}=\Lambda(\bJ)(\mathfrak{L}_0)$ is a~$\bJ$-stable lattice in~$V_\Lambda$ by~\cite[9.3]{V}.   The induced~$\Zl$-representation~$\ind_\bJ^G(\mathfrak{L}_{\Lambda})$ is then a lattice in~$\pi$ by~\cite[9.3]{V}. Moreover~$\tau=r_\ell(\pi)\simeq \ind_{\bJ}^G(r_\ell(\Lambda))$, and~$r_\ell(\Lambda)$ is an irreducible representation of~$\bJ$ by irreducibility of~$\tau$. 

Finally, we give another characterisation of banal cuspidal representations: Recall, from Section \ref{coeffinR}, by definition~$\tau$ is banal if and only if the cardinality of the cuspidal line 
$\Z_{\tau}=\{ |~|^k \tau,\ k\in \Z\}$ is greater than~$1$. By \cite[Lemme 5.3]{MSDuke}, this cardinality is the same as the integer~$o(\tau)$ introduced in \cite[Section 5.2, (5.4)]{MSDuke}. From \cite[Section 5.2, (5.4)]{MSDuke},~$o(\tau)$ is the order of~$q^{n/e}$ in~$\Fl^\times$, where~$e=e(E/F)$ is the ramification index attached to~$(\bJ, \Lambda)$ which in particular does not depend on the choice of extended maximal simple type. Hence~$\tau$ is banal if and only if~$q^{n/e}-1\neq 0$ in~$\Fl$.

\section{The modified Paskunas-Stevens basis}\label{PS-basis}
\textit{For this section~$R=\mathbb{C}$ or~$\Ql$}.  Let~$\pi$ be a cuspidal~$R$-representation of~$G$ and~$(\bJ=\bJ(\beta,\mathcal{L}),\Lambda)$ be an extended maximal simple type in~$\pi$.  According to~\cite[Corollaries~3.4~and~4.13]{PS}, there exists an~$F$-basis~$\mathcal{B}=(v_1,\dots,v_n)$ of~$V$ particularly suited to relating the Whittaker model of~$\pi$ and the model~$\ind_{\bJ}^G(\Lambda)$ defined via type theory. In particular,~$\mathcal{B}$ splits~$\mathcal{L}$, i.e.~$L_k=\bigoplus_{i=1}^n \p_F^{a_i(k)} v_i$ with~$a_i(k)\in \Z$ for all~$k\in \Z$, and is such that if~$N$ is the maximal unipotent subgroup of~$G$ attached to the maximal flag defined by~$\mathcal{B}$, and if~$\psi$, by abuse of notation, denotes the non-degenerate character of~$N$ defined for~$x\in N$ by 
\[\psi(x)= \psi\left(\sum_{i=1}^{n-1} \Mat_\mathcal{B}(x)_{i,i+1}\right),\] 
where~$\Mat_\mathcal{B}(x)$ denotes the matrix of~$x$ with respect to the basis~$\mathcal{B}$, then the triple~$(J,\Lambda,\psi)$ satisfies \[\Hom_{N\cap J}(\psi,\Lambda)\neq 0.\]

Let~$P$ be the \emph{mirabolic subgroup} defined by
\[P=\{g\in G, (g-Id)V \subset \Vect_F(v_1,\dots,v_{n-1})\}.\]
We put~$\mathcal{M}=(P\cap J)J^1$, which is a group as~$J^1$ is normal in~$J$.  It follows from \cite{PS} that the image of~$\mathcal{M}$ in~$J/J^1\simeq G_m(k_E)$ is isomorphic to~$P_m(k_E)$.  We now explain how to extract this from \cite{PS}:  In the notation of \cite{PS}, our group~$P$ is denoted~$\mathcal{M}_F$ and~\cite[Corollary~4.8]{PS} shows that
\begin{equation}\label{PSgroupsequation}\mathcal{M}= (P\cap \mathfrak{B}^\times)J^1.\end{equation} In \cite[Section 4.1]{PS}, Paskunas--Stevens introduce another mirabolic group they denote by~$\mathcal{M}_E$ which satisfies~$P\cap \mathfrak{B}^\times=\mathcal{M}_E\cap \mathfrak{B}^\times$ by the equality just before~\cite[Corollary~4.7]{PS}, and they also denote by~$\mathcal{M}_\mathfrak{B}$ the group~$(\mathcal{M}_E\cap \mathfrak{B}^\times)(1+\w_E \mathfrak{B})$.  Hence Equation (\ref{PSgroupsequation}) gives~$\mathcal{M}=\mathcal{M}_\mathfrak{B}J^1$ as~$(1+\w_E \mathfrak{B})= \mathfrak{B}^\times\cap J^1$. Finally, from the discussion after the proof of~\cite[Lemma~4.10]{PS}, the image of~$\mathcal{M}_\mathfrak{B}$ in~$\mathfrak{B}^\times /1+\w_E \mathfrak{B}\simeq G_m(k_E)$  is isomorphic to~$P_m(k_E)$, hence the same is true for the image of~$\mathcal{M}$ in~$J/J^1\simeq \mathfrak{B}^\times /1+\w_E \mathfrak{B}\simeq G_m(k_E)$. In particular, the following index will appear in our computation:
\[|J/\mathcal{M}|=|G_m(k_E)/P_m(k_E)|=q_E^m-1=q^{n/e}-1.\]

For~$i\in\{1,\ldots,n\}$, the functions~$a_i:\mathbb{Z}\rightarrow\mathbb{Z}$ satisfy the relation~$a_{i+e}(k)=a_i(k)+1$. In particular, this holds for~$i=n$, and the map~$k\mapsto a_n(k)$ is increasing with values in~$\Z$, so there is~$k_0$ between~$1$ and~$e$ such that~$a_n(k_0)=a_n(k_0-1)+1$, and then~$a_n(k_0+i)=a_n(k_0)$ for~$i\in \{0,\ldots ,e-1\}$.  Hence by reindexing the lattice chain~$\mathcal{L}$ if necessary, by a translation,~$k\mapsto k-k_0$, we can suppose that
 \[a_n(0)=a_n(-1)+1=0,\text{ and }a_n(1)=\dots=a_n(e-1)=0.\]   We recall that~$L_0=\bigoplus_{i=1}^n \p_F^{a_i(0)} v_i$, and we set~$\mathcal{B}'=(\w_F^{a_1(0)}v_1,\dots,\w_F^{a_n(0)}v_n)$, which we write as~$\mathcal{B}'=(w_1,\dots,w_n)$.  
 
We use this basis to identify~$G$ with~$G_n$. With this choice, one has~$J\subset K_n$ because~$J
\subset \Aut_{\of_F}(L_0)$. The group~$P$ identifies with~$P_n$, the group~$N$ identifies with~$N_n$, and the character~$\psi$ of $N_n$ identifies with 
\[\psi_t: n\mapsto \psi\left(\sum_{i=1}^{n-1} t_i n_{i,i+1}\right),\] where~$t_i=\w_F^{a_i(0)-a_{i+1}(0)}$.

For our computation to come, it will be useful to notice the following property of~$\mathcal{B'}$: 
one has~\[L_0=\bigoplus_{i=1}^n \of_F w_i\text{,\quad\quad }L_k=\bigoplus_{i=1}^{n-1}\p_F^{a_i(k)-a_i(0)}w_i\oplus \of_F w_n,\] for~$k\in \{1,\dots,e-1\}.$  As~$\w_E L_k= L_{k+1}$ for any~$k\in\mathbb{Z}$, the properties above and the fact that~$L_{k+e}=\w_F L_k$, imply that the last row of~$\w_E^i\in G_n$ belongs to~$(\of_F)^n-(\p_F)^n$ for~${i=0,\dots,e-1}$, and more generally that it belongs to~$(\p_F^l)^n-(\p_F^{l+1})^n$ if~$i=le+r$, with~$r\in \{0,\dots,e-1\}$. As an immediate consequence, 
if we write an Iwasawa decomposition of~$\w_E^i$,
\[\w_E^i=p_iz_ik_i,\qquad p_i\in P_n,~ z_i\in Z_n,~ k_i\in K_n,\] 
 we can choose~$z_i=I_n$ for~$i=0,\dots,e-1$, and more generally~$z_i=\w_F^lI_n$ for~$i=le+r$, with~$r\in \{0,\dots,e-1\}$.  In particular~$|p_i|=q^{-in/e}$, for~$i=0,\dots,e-1$. 
 
For clarity, we list the properties of the data~$(J,\Lambda,\psi_t)$ that we will use.
\begin{prop}\label{properties}
With the above choice of basis we have:
\begin{enumerate}
\item The inclusion~$J\subset K_n$.
\item The space~$\Hom_{N_n\cap J}(\psi_t,\Lambda)\neq 0.$
\item  \label{properties3} Set~$\mathcal{M}=(P_n\cap J)J^1$, then~$|J/\mathcal{M}|=q^{n/e}-1$.
\item\label{Prop314} The element~$\w_E^i\in P_n K_n$ if and only if~$i\in \{0,\dots,e-1\}$ and, in this case, if we choose~$p_i\in P_n$ and~$k_i\in K_n$, such that~$\w_E^i=p_ik_i$, then we have~$|p_i|=|\w_E^i|=q^{-in/e}.$
\end{enumerate}
\end{prop}
For the remainder, we consider the~$k_i\in K_n$ and~$p_i\in P_n$ chosen in Proposition~\ref{properties} Statement~\ref{Prop314} as fixed.  

As~$P_n\cap J^1$ is a pro-$p$ sugbroup of~$P_n$, and~$J^1$ is a pro-$p$ sugbroup of~$G_n$, the volume 
\[dk(P_n\cap J^1\backslash J^1)=\frac{dk(J^1)}{dp(P_n\cap J^1)}\] is a power of~$q$ thanks to our normalisation of measures, and we write
 \[dk(P_n\cap J^1\backslash J^1)=q^{r_1}.\]
 A certain volume will appear in our later computation, we compute it in the next lemma. 

\begin{LM}\label{volume}
For any~$i\in \{0,\dots,e-1\}$, we have
\[dk((P_n\cap K_n)\backslash (P_n\cap K_n) k_i J)=q^{r_1} (q^{n/e}-1)q^{-in/e}.\]
\end{LM}
\begin{proof}
We have
\begin{align*}
dk((P_n\cap K_n)\backslash (P_n\cap K_n) k_i J)&=dk((P_n\cap K_n)\backslash (P_n\cap K_n) k_i Jk_i^{-1})\\&=
dk((P_n\cap  k_i Jk_i^{-1})\backslash k_i Jk_i^{-1}),
\end{align*}
the last equality thanks to Relation (\ref{equal}).  
Now,~$dk( k_i Jk_i^{-1})=dk(J)$. We also notice that 
\[p_i(P_n\cap  k_i Jk_i^{-1})p_i^{-1}=P_n\cap \w_E^i J \w_E^{-i}=P_n\cap J,\] 
hence \[P_n\cap  k_i Jk_i^{-1}=p_i^{-1}(P_n\cap J)p_i.\]
As for any compact open subset~$A$ of~$P_n$, one has~$dp(pAp^{-1})=|p|dp(A)$, as is easily seen by writing~$dp=dgdu$, with~$dg$ on~$G_{n-1}$ 
and~$du$ on~$U_n$, we obtain the relation 
\[dp(P_n\cap  k_i Jk_i^{-1})=|p_i|^{-1}dp(P_n\cap J)=q^{in/e}dp(P_n\cap J).\] We then obtain from Relations (\ref{quotient}) and (\ref{cardinal}):
\begin{align*}
dk((P_n\cap  k_i Jk_i^{-1})\backslash k_i Jk_i^{-1})&=\frac{dk( k_i Jk_i^{-1})}{dp(P_n\cap  k_i Jk_i^{-1}))}\\
&=
q^{-in/e}\frac{dk( J)}{dp(P_n\cap J))}\\&=
q^{-in/e}dk((P_n\cap   J)\backslash J).
\end{align*}
Now by Relations (\ref{quotient}) and (\ref{cardinal}) again, one has 
\[dk((P_n\cap   J)\backslash J))=\frac{dk(J)}{dp(P_n\cap J)}=
\frac{dk(J)}{dk(\mathcal{M})}\frac{dk(\mathcal{M})}{dp(P_n\cap J)}= 
|J\backslash \mathcal{M}|dk(P_n\cap J \backslash \mathcal{M}).\]
Finally, because~$\mathcal{M}=(P_n\cap   J)J^1$, applying Relation (\ref{equal}) gives:
\[dk((P_n\cap   J)\backslash J))=|J\backslash \mathcal{M}|dk(P_n\cap J^1 \backslash J^1)=q^{r_1}(q^{n/e}-1)\]
by Proposition \ref{properties} \ref{properties3} and our definition of~$r_1$.  This concludes the proof.
\end{proof}

\section{Explicit Whittaker functions of Paskunas--Stevens}\label{SectWhittaker}
In this Section we continue to assume that~$R=\mathbb{C}$ or~$\Ql$.  We now recall the definition and some properties of the explicit Whittaker functions of~\cite{PS}.  We set 
\[\mathcal{U}=(N_n\cap J)H^1.\]
We extend~$\psi_t$ to the group~$\mathcal{U}$ as in~\cite[Definition~4.2]{PS}, and, by abuse of notation, denote this extension by~$\psi_t$. We fix a normal compact open subgroup~$\mathcal{N}$ of~$\mathcal{U}$ contained in~$\ker(\psi_t)$. We also denote by~$\rho$ the trace character of~$\Lambda$ and~$\rho^\vee$ that of~$\Lambda^\vee$.

\begin{df}[\emph{Bessel functions}]\label{Bessel}
For~$j \in \bold{J}$, we define \[\mathcal{J}(j)=|\mathcal{N}\backslash \mathcal{U}|^{-1}\sum_{\mathcal{N}\backslash \mathcal{U}} \psi_t(u)^{-1}\rho(ju),\text{ and }\mathcal{J}^\vee(j)= |\mathcal{N}\backslash \mathcal{U}|^{-1}\sum_{\mathcal{N}\backslash \mathcal{U}} \psi_t(u)\rho^\vee(ju).\]
\end{df}

The Bessel functions enjoy the following properties:

\begin{prop}\label{bessel}
\begin{enumerate}
\item We have the equality~$\mathcal{J}(1)=1$.
\item \label{bessel1}~$\mathcal{J}(uj)=\mathcal{J}(ju)=\psi_t(u)\mathcal{J}(j)$ for~$u\in \mathcal{U}$ and 
$j\in \bold{J}$.
\item \label{dual-bessel} For all~$j\in \bold{J}$, we have the relation
\[\mathcal{J}^\vee(j)=\mathcal{J}(j^{-1}).\]
\item \label{bessel2} For all~$j_1$ and~$j_2$ in~$\bold{J}$, we have
\[\sum_{m\in \mathcal{U}\backslash \mathcal{M}} \mathcal{J}  (j_1m^{-1}) \mathcal{J} (mj_2) = \mathcal{J}(j_1j_2).\]\end{enumerate}
\end{prop}

\begin{proof}
See~\cite[Proposition~5.3~and~Theorem~5.6] {PS}. The third property follows from a simple change of variables, and the relation~$\rho^{\vee}(ab)=\rho(b^{-1}a^{-1})$ for any~$a$ and~$b$ in~$\bold{J}$. The final property follows from \cite[Proposition~5.3, Property~(v)]{PS}, thanks to the bijection~$m\leftrightarrow m^{-1}$ between~$\mathcal{M}/ \mathcal{U}$ and~$\mathcal{U}\backslash \mathcal{M}$.
\end{proof}

We can now define the explicit Whittaker functions~$W$ and~$W^\vee$ of Paskunas--Stevens following~\cite[Section~5.2]{PS} and recall a first property. 

\begin{df}\label{ExpWhits}
Both~$W$ and~$W^\vee$ are supported on~$N_n \bold{J}$, and \[W(nj)=\psi_t(n)\J(j)\] for~$n\in N_n$ and~$j\in \bold{J}$, whereas 
\[W^\vee(nj)=\psi_t^{-1}(n)\J^\vee(j)=\psi_t^{-1}(n)\J(j^{-1})\] for~$n\in N_n$ and~$j\in \bold{J}$. Moreover,~$W$ 
belongs to~$W(\pi,\psi_t)$ and~$W^\vee$ belongs to~$W(\pi^\vee,\psi_t^{-1})$.
\end{df}

We now prove further properties of~$W$ and~$W^\vee$. 

\begin{prop}\label{furtherproperties} 
For~$l\geqslant 0$, let~$W_l=\mathbf{1}_{G_n^{(l)}}W,$ and~$W_l^\vee=\mathbf{1}_{G_n^{(l)}}W^\vee$.
\begin{enumerate}
\item\label{point1}  The functions~$(W_l)\mid_{P_nK_n}$ and~$(W_l)^\vee\mid_{P_nK_n}$ are zero unless~$l=in/e$ for some~$i\in\{0,\dots,e-1\}$, and in this case 
\[(W_l)\mid_{P_nK_n}=\mathbf{1}_{N_n\w_E^iJ}W\mid_{P_nK_n} \text{,\qquad} (W_l^\vee)\mid_{P_nK_n}= 
\mathbf{1}_{N_n\w_E^iJ}W^\vee\mid_{P_nK_n}.\]
\item\label{point2} If~$W_{in/e}(pk)\neq 0$, then~$i\in\{0,\dots,e-1\}$,~$k\in P_n\w_E^iJ$, and, in fact,~$k\in (P_n\cap K_n)k_iJ$.
\item\label{point3} If~$W_{in/e}(p\w_E^ij)\neq 0$ with~$p\in P_n$ and~$j\in J$, then~$p\in N_n(P_n\cap J)$. \end{enumerate}
\end{prop}
\begin{proof}
The first statement follows from the fact that~$W$ is supported on~$N_n\bold{J}=\coprod_{i\in \Z} N_n \w_E^i J$, this is a disjoint union because the absolute value of the determinant on~$N_n \w_E^i J$ is~$q^{-ni/e}$, and 
Statement~\ref{Prop314} of Proposition~\ref{properties}.  Hence, if~$W_{in/e}(pk)\neq 0$, then~$W(pk)\neq 0$, so~$pk\in N_n\w_E^lJ$ 
for a unique~$l\in \{0,\dots,e-1\}$, but this~$l$ must be equal to~$i$, and this gives the first assertion of the second statement. In particular~$k\in p^{-1}N_n\w_E^iJ\subset P_n\w_E^iJ$. But~$P_n\w_E^iJ=P_np_ik_iJ=P_nk_iJ$, hence~$k\in P_nk_iJ\cap K_n=(P_n\cap K_n)k_iJ$. This proves the second statement. For the third, we observe that if~$W_{in/e}(p\w_E^ij)\neq 0$, then~$p\w_E^ij\in N_n \w_E^i J$, hence~$p\in N_n \w_E^i Jj^{-1}\w_E^{-i}=N_nJ$, which implies that~$p\in N_n(P_n\cap J)$.
\end{proof}

\section{Test vectors}\label{Main}

Again,~we assume that~$R=\C$, or~$\Ql$, and~$\pi_1$ and~$\pi_2$ are cuspidal~$R$-representations of~$G_n$. We denote by~$(\bJ,\Lambda)$ the extended maximal simple type of~$\pi_1$, by~$e=e(E/F)$ the ramification index of the field extension associated to~$(\bJ,\Lambda)$, and by~$W,W^{\vee}$ the explicit Whittaker functions associated to~$\pi_1$ (see Definition~\ref{ExpWhits}). This section is dedicated to proving our main result on test vectors. 

\begin{thm}\label{main}
Suppose that~$L(X,\pi_1,\pi_2)$ is non-trivial, so that~$\pi_2\simeq \chi \pi_1^\vee$ for some unramified character~$\chi$ of~$F^\times$. Then there is an integer~$r$ such that 
\begin{equation*}\label{test}I(X,W,\chi W^\vee,1_{\of_F^n})=\frac{q^r(q-1)(q^{n/e}-1)}{1-(\chi(\w_F)X)^{n/e}}=q^r(q-1)(q^{n/e}-1)L(X,\pi_1,\pi_2).\end{equation*} 
\end{thm} 

We are now ready to prove the following crucial proposition. We recall that for all integers~$l\geqslant 0$, the restriction~$W_l$ has been defined in Proposition \ref{furtherproperties}.

\begin{prop}\label{F_i}
Let~$F_l:(K_n\cap P_n)\backslash K_n/J^1\rightarrow R$ be defined by 
\[F_l(k)=\int_{j\in J^1}\int_{N_n\backslash P_n} W_l(pkj)W_l^\vee(pkj) dpdj.\]
Then~$F_l$ is nonzero if an only if~$l=in/e$ and~$i\in \{0,\dots,e-1\}$, and in this case, it is supported on~$(K_n\cap P_n)k_i J$.  Moreover, for~$i\in \{0,\dots,e-1\}$, and for~$k\in (K_n\cap P_n)k_i J$, there is an integer~$r_2$ independent of~$i$ such that
\[F_{in/e}(k)=q^{r_2}.\]
\end{prop}
\begin{proof}
If~$F_l(k)$  is nonzero, then~$W_l(pkj)$ is nonzero at least for some~$p\in P_n$ and~$j\in J$, but then according to Statements~\ref{point1} and~\ref{point2} of Proposition~\ref{furtherproperties}, 
this implies that~$l$ is of the form~$l=in/e$ with~$i\in \{0,\dots,e-1\}$, and~$k\in (K_n\cap P_n)k_i J$. Moreover, from Statement~\ref{point2} of the same proposition, we can write~$k=p_0\w_E^i j_0$ for~$p_0\in P_n$ and~$j_0\in J$. But now notice that for such a~$k$, we have
\begin{align*}F_l(k)&=\int_{j\in J^1}\int_{N_n\backslash P_n} W_l(pp_0\w_E^i j_0j)W_l^\vee(pp_0\w_E^i j_0j) dpdj\\
&=\int_{j\in J^1}\int_{N_n\backslash P_n} W_l(p\w_E^i j_0j)W_l^\vee(p\w_E^i j_0j) dpdj.\end{align*}
hence by Statement~\ref{point3} of Proposition~\ref{furtherproperties}
\begin{align*}F_l(k)&=\int_{j\in J^1}\int_{N_n\backslash N_n(P_n\cap J)} W_l(p\w_E^i j_0j)W_l^\vee(p\w_E^i j_0j) dpdj\\
&=\int_{j\in J^1}\int_{N_n\cap J\backslash P_n\cap J} W_l(m\w_E^i j_0j)W_l^\vee(m\w_E^i j_0j) dmdj\\
&=\int_{j\in J^1}\int_{N_n\cap J\backslash P_n\cap J} \mathcal{J}(m\w_E^i j_0j) \mathcal{J}(j^{-1} j_0^{-1}\w_E^{-i} m^{-1}) dmdj,\end{align*} the last 
equality according to Proposition \ref{bessel}~\ref{dual-bessel}. Now, as~$\bold{J}$ normalises~$J^1$, and as for any~$t\in G_n$ normalising~$J^1$, the automorphism~$j\mapsto tjt^{-1}$ of~$J^1$ has modulus character equal to~$\1$, because~$J^1$ is an open subgroup of the unimodular group~$G_n$, we have 
\begin{align*}
F_l(k)&=\int_{j\in J^1}\int_{N_n\cap J\backslash P_n\cap J} \mathcal{J}(mj\w_E^i j_0) \mathcal{J}(j_0^{-1}\w_E^{-i}(mj)^{-1}) dmdj\\
&=\int_{N_n\cap J\backslash \mathcal{M}} \mathcal{J}(m\w_E^i j_0) \mathcal{J}(j_0^{-1}\w_E^{-i}m^{-1}) dm.
\end{align*}
 We write 
 \[dm(N_n \cap J\backslash (N_n\cap J)H^1)=dm(N_n \cap H^1\backslash H^1)=q^{r_2},\] 
 which is indeed a power of~$q$ as~$H^1$ is pro-$p$.  Moreover, as~$H^1$ is normal in~$J$, and as the integrand is invariant under~$\mathcal{U}$ thanks to Property \ref{bessel1} in Proposition \ref{bessel}
\[F_l(k)=q^{r_2} \int_{\mathcal{U}\backslash \mathcal{M}} \mathcal{J}(m\w_E^i j) \mathcal{J}(j^{-1}\w_E^{-i}m^{-1}) dm=
q^{r_2},\]
the last equality thanks to Statement~\ref{bessel2} of Proposition~\ref{bessel}.  
\end{proof}

\begin{prop}\label{coef}
The coefficient 
\[b_l=\int_{P_n\cap K_n\backslash K_n}\int_{N_n\backslash P_n} W_l(pk)W_l^\vee(pk)dpdk\]
is zero unless~$l=in/e$ for some~$i\in \{0,\dots,e-1\}$, in which case there is an integer~$r$ such that \[b_l=q^r(q^{n/e}-1) q^{-in/e}.\]
\end{prop}
\begin{proof}
By definition,~$b_l$ is equal to 
\[\int_{P_n\cap K_n\backslash K_n/J^1}F_l(k)dk=q^{r_3} \int_{P_n\cap K_n\backslash K_n}F_l(k)dk\] with~$dk(J^1)=q^{r_3}$ ($J^1$ is pro-$p$). So according to Proposition~\ref{F_i}, this is zero if~$l\neq in/e$ for~$i\in\{0,\dots,e-1\}$, and if~$l= 
in/e$ for~$i\in\{0,\dots,e-1\}$, it is equal to 
\begin{align*}
q^{r_3}\int_{P_n\cap K_n\backslash (P_n\cap K_n)k_i J}F_l(k)dk&=q^{r_2+r_3} dk(P_n\cap K_n\backslash (P_n\cap K_n)k_i J)\\&= q^{r}(q^{n/e}-1) q^{-in/e},\end{align*}
where we write~$r=r_1+r_2+r_3$, from Lemma~\ref{volume}. 
\end{proof}

If~$\pi$ is a cuspidal~$R$-representation of~$G_n$ of ramification index~$e$, we denote by~$R(\pi)$ its \emph{ramification group}, that is the group of unramified characters~$\nu$ of~$F^\times$ which satisfy~$\nu\pi\simeq \pi$. It follows from~\cite[6.2.5]{BK93}, that~$R(\pi)$ is isomorphic to the group of~$n/e$-th roots of unity in~$R^\times$, via~$\nu\mapsto \nu(\w_F)$. 

\begin{proof}[Proof of Theorem~\ref{main}]
We first suppose that~$\pi_2\simeq \pi_1^\vee$. By Equation (\ref{I=LI_{(0)}}), the integral~$I(X,W,W^\vee,\1_{\of_F^n})$ is equal to 
\[\frac{q-1}{1-X^n}\int_{(K_n\cap P_n)\backslash K_n} I_{(0)}(X,\rho(k)W,\rho(k)W^\vee)dk.\]
 Now, as~$W W^\vee=\sum_{l\in \Z} W_lW_l^\vee$, by Statement~\ref{point1} of Proposition~\ref{furtherproperties}, and Proposition 
\ref{coef}, we have 
\begin{align*}
\int_{(K_n\cap P_n)\backslash K_n} I_{(0)}(X,\rho(k)W,\rho(k)W^\vee)dk&=\sum_{i=0}^{e-1} b_{in/e} q^{in/e} X^{in/e}\\
= q^r(q^{n/e}-1)\sum_{i=0}^{e-1} X^{in/e}&=q^r(q^{n/e}-1)\frac{1-X^n}{1-X^{n/e}}.\end{align*} 
This gives the equality 
\[I(X,W,W^\vee,\1_{\of_F^n})= (q-1)(q^{n/e}-1)\frac{q^r}{1-X^{n/e}}.\]
On the other hand, and by~\cite[Proposition~8.1]{JPS2}, the 
factor~$L(X,\pi,\pi^{\vee})$ is equal to 
\[L(X,\pi,\pi^{\vee})=\prod_{\nu\in R(\pi)} \frac{1}{1-\nu(\w_f)X}=\frac{1}{1-X^{n/e}}.\]

Now in general, as we supposed that~$L(X,\pi_1,\pi_2)$ is not equal to~$1$, we have~$\pi_2\simeq \chi \pi_1^\vee$ for some unramified character~$\chi$ of~$F^\times$. However, we have 
\[L(X,\pi_1,\pi_2)=L(X,\pi_1,\chi \pi_1^\vee)=L(\chi(\w_F)X,\pi_1,\pi_1^\vee).\]
On the other hand, we have
\begin{align*}I(X,W,\chi W^\vee,\1_{{\of_F^n}})&=I(\chi(\w_F)X,W, W^\vee,\1_{{\of_F^n}})\\
&=(q-1)(q^{n/e}-1)\frac{q^r}{1-(\chi(\w_F)X)^{n/e}}.\end{align*}
However,
\begin{align*}L(X,\pi_1,\pi_2)=L(\chi(\w_F)X,\pi,\pi^\vee)
=\frac{1}{1-(\chi(\w_F)X)^{n/e}},\end{align*} and we are done.
\end{proof}


\section{$L$-factors of banal cuspidal~$\ell$-modular representations}
In this section, we consider the cases~$R=\Fl$, and~$R=\Ql$.  In the~$\Ql$ setting, we continue with the notations of the last section, and note that as~$\psi$ is integral, so are~$\psi_t$ and~$\psi_t^{-1}$.    Our main theorem has the following interesting corollary. 

\begin{cor}\label{lmodcor}
Let~$\tau_1$ and~$\tau_2$ be two banal cuspidal~$\ell$-modular representations of~$G_n$, and~$\pi_1$ and~$\pi_2$ be any cuspidal~$\ell$-adic lifts, then
\[L(X,\tau_1,\tau_2)=r_\ell(L(X,\pi_1,\pi_2)).\]
\end{cor}

\begin{proof}
We already noticed in Section \ref{local factors} that if~$L(X,\pi_1,\pi_2)$ is equal to~$1$, then 
\[L(X,\tau_1,\tau_2)=r_\ell(L(X,\pi_1,\pi_2))= 1,\]
whether~$\tau_1$ and~$\tau_2$ are banal or not. Hence we only need to focus on the case when~$L(X,\pi_1,\pi_2)$ is not equal to~$1$, that is~$\pi_2\simeq \chi\pi_1^\vee$ for some unramified character~$\chi$. Let~$W$ be the Stevens-Paskunas explicit Whittaker function associated to an extended maximal simple type of~$\pi_1$ as in the statement of Theorem~\ref{main}.
\begin{LM}
The explicit Whittaker functions~$W$ and~$\chi W^\vee$ lie in the~$\Zl$-submodules~$W_e(\pi_1,\psi_t)$ and~$W_e(\pi_2,\psi_t^{-1})$ respectively.
\end{LM}
\begin{proof}
As in the proof of Theorem~\ref{main}, the representation~$\pi_1$ contains an extended maximal simple type~$(\bJ_1,\Lambda_1)$ and~$W$ is chosen to be the Paskunas--Stevens Whittaker function of Definition~\ref{ExpWhits} relative to this data.  As~$\pi_1$ is integral,~$\Lambda_1$ is integral by the end of Section \ref{SimpleTypes}.  This implies that the trace character~$\rho_{\Lambda_1}$ of~$\Lambda_1$ has values in~$\Zl$.  In particular the Bessel function~$\mathcal{J}_1$ (see Definition~\ref{Bessel}) associated to the pair 
$(\bJ_1,\Lambda_1)$ takes values in~$\Zl$. Hence, as~$\psi_t$ is integral,~$W\in W_e(\pi_1,\psi_t)$ (see Definition~\ref{ExpWhits}). Now,~$\pi_2$ is of the form~$\chi\pi_1^{\vee}$ with~$\chi$ an unramified character of~$F^\times$ (which is integral as~$\chi$ is unramified), so Proposition \ref{bessel} \ref{dual-bessel} implies that the Bessel function 
$\chi\mathcal{J}_1^\vee$ is integral. We  
conclude that~$\chi W^\vee$ belongs to~$W_e(\pi_2,\psi_t^{-1})$ (see Definition~\ref{ExpWhits} again).
\end{proof}
Granted~$W\in W_e(\pi_1,\psi_t)$ and~$\chi W^\vee\in W_e(\pi_2,\psi_t)$, we have 
\begin{align*}r_\ell(q^r(q^{n/e}-1))r_\ell(L(X,\pi_1,\pi_2)) &=r_\ell(I(X,W,\chi W^\vee,\1_{\of_F^n}))\\&= 
I(X,r_\ell(W),r_\ell(\chi W^{\vee}),r_\ell(\1_{\of_F^n})).\end{align*} 
Notice that~$r_\ell(q^r(q-1)(q^{n/e}-1))$ is nonzero if and only if~$\pi_1$ (hence~$\pi_2$) is banal by the end of Section \ref{SimpleTypes}. As the integral~$I(X,r_\ell(W),r_\ell(\chi W^\vee),r_\ell(\1_{\of_F^n}))$ belongs 
to the fractional ideal~$(L(X,\tau_1,\tau_2))$ of~$\Fl[X^{\pm1}]$, we deduce that $r_\ell(L(X,\pi_1,\pi_2))$ divides~$L(X,\tau_1,\tau_2)$. As in any case, thanks to~\cite[Theorem~3.13]{KM14}, the~$L$-factor~$L(X,\tau_1,\tau_2)$ divides~$r_\ell(L(X,\pi_1,\pi_2))$, we deduce the desired equality.  
\end{proof}

\begin{rem} As noticed in the introduction and Section \ref{local factors}, the analogue of Corollary \ref{lmodcor} is also true when~$\pi_1$ and~$\pi_2$ are cuspidal representations of general linear groups of different ranks as the~$L$-factors are all trivial. 
\end{rem}

\begin{acknowledgements}
We thank Paul Broussous and Shaun Stevens for stimulating discussions, and Jim Cogdell for bringing the test vector problem to our attention.  We thank the referee for a very thorough review which improved greatly the presentation and clarified certain parts of the arguments.  This work has benefited from support from the Heilbronn Institute for Mathematical Research, GDRI 571 ``French-British-German network in Representation Theory'', and from the grant ANR-13-BS01-0012 FERPLAY.
\end{acknowledgements}

\bibliographystyle{plain}
\bibliography{Modlfactors}

\end{document}